\documentclass[12pt]{amsart}

\usepackage{amsmath,amssymb,amsfonts,amsthm}
\usepackage{mathtools}
\usepackage[margin=1in]{geometry}

\usepackage{graphicx}
\graphicspath{ {./figures/} } 
\usepackage{tikz}
\usetikzlibrary{arrows.meta,graphs,shapes.geometric,decorations.pathreplacing}
\usepackage{pgfplots}
\pgfplotsset{compat=1.18}

\usepackage{booktabs}
\usepackage{multicol}
\usepackage{enumitem}

\usepackage[numbers,sort&compress]{natbib}

\usepackage{xcolor}
\definecolor{myteal}{RGB}{0,128,128}
\usepackage[pdfborder={0 0 0}]{hyperref}
\hypersetup{
  colorlinks=true,
  citecolor=myteal,
  linkcolor=blue,
  urlcolor=myteal
}

\usepackage{aliascnt}
\usepackage[capitalise, noabbrev]{cleveref}

\usepackage{parskip}

\usepackage{listings}
\usepackage{booktabs}
\newcommand{\newaliastheorem}[3]{%
  \newaliascnt{#1}{thm}%
  \newtheorem{#1}[#1]{#2}%
  \aliascntresetthe{#1}%
  \crefname{#1}{#2}{#3}%
  \Crefname{#1}{#2}{#3}%
}

\theoremstyle{plain}
\newtheorem{thm}{Theorem}[section]
\crefname{thm}{Theorem}{Theorems}
\Crefname{thm}{Theorem}{Theorems}

\newaliastheorem{lem}{Lemma}{Lemmas}
\newaliastheorem{prop}{Proposition}{Propositions}
\newaliastheorem{cor}{Corollary}{Corollaries}
\newaliastheorem{clm}{Claim}{Claims}
\newaliastheorem{conj}{Conjecture}{Conjectures}

\theoremstyle{definition}
\newaliastheorem{defn}{Definition}{Definitions}
\newaliastheorem{ex}{Example}{Examples}
\newaliastheorem{notation}{Notation}{Notations}
\newaliastheorem{setup}{Set-up}{Set-ups}
\newaliastheorem{remark}{Remark}{Remarks}
\newaliastheorem{problem}{Problem}{Problems}
\newaliastheorem{question}{Question}{Questions}
\newaliastheorem{observation}{Observation}{Observation}

\setlist[enumerate]{leftmargin=*,itemsep=2pt,topsep=4pt}
\SetEnumerateShortLabel{a}{\textup{(\alph*)}}
\SetEnumerateShortLabel{A}{\textup{(\Alph*)}}
\SetEnumerateShortLabel{1}{\textup{(\arabic*)}}
\SetEnumerateShortLabel{i}{\textup{(\roman*)}}
\SetEnumerateShortLabel{I}{\textup{(\Roman*)}}

\makeatletter
\newtheorem*{rep@theorem}{\rep@title}
\newcommand{\newreptheorem}[2]{%
\newenvironment{rep#1}[1]{%
 \def\rep@title{#2 \ref{##1}}%
 \begin{rep@theorem}}%
 {\end{rep@theorem}}}
\makeatother

\newreptheorem{theorem}{Theorem}

\begin{document}

\title{Symmetric and unimodal independence polynomials \\ of trees}

\author[Hibi]{Takayuki Hibi}
\address[T.~Hibi]{Department of Pure and Applied Mathematics, Graduate School of Information Science and Technology, Osaka University, Suita, Osaka 565--0871, Japan}
\email{\href{mailto:hibi@math.sci.osaka-u.ac.jp}{hibi@math.sci.osaka-u.ac.jp}}

\author[Kara]{Selvi Kara}
\address[S.~Kara]{Department of Mathematics, Bryn Mawr College, Bryn Mawr, PA 19010}
\email{\href{mailto:skara@brynmawr.edu}{skara@brynmawr.edu}}

\author[Vien]{Dalena Vien}
\address[D.~Vien]{Department of Mathematics, Bryn Mawr College, Bryn Mawr, PA 19010}
\email{\href{mailto:dvien@brynmawr.edu}{dvien@brynmawr.edu}}

	

% 05C25  	Graphs and abstract algebra 
% 05C30  	Enumeration in graph theory
% 05C31  	Graph polynomials
% 05C38  	Paths and cycles [See also 90B10]
    	   	
	\thanks{The research for the present paper was initiated while the first author visited Bryn Mawr College, Pennsylvania, February 28 -- March 21, 2026. } 

\begin{abstract}
   Given $n \geq 1$, we study the existence of a tree on $n$ vertices whose independence polynomial is symmetric and unimodal as well as the existence of a symmetric and unimodal independence polynomial of degree $n$ of a tree. 
\end{abstract}

\maketitle

\section*{Introduction}
Every graph considered in this paper is finite and simple. For a graph \(G\), let
\(g_i(G)\) denote the number of independent sets of \(G\) of cardinality \(i\). The
\emph{independence polynomial} of \(G\) is
\[
P_G(x)=\sum_{i\geq 0} g_i(G)x^i.
\]
Equivalently, independence polynomials are the face polynomials of independence
complexes of graphs, and are therefore closely related to the study of face vectors of flag
complexes. They have been studied extensively from several points of view
\cite{Levit_Mandrescu,Unimodal_ind_poly,location_roots_ind_poly,location_complex_roots,Claw_free_ind_poly} in graph theory.

In 1987 \cite{tree_unimodal_conj}, Alavi, Malde, Schwenk, and Erd\H{o}s conjectured  that the independence
polynomials of trees are unimodal. As of April 2026, the conjecture remains open in general, and has been computationally verified for trees on at most 
29 vertices \cite{reynolds2026mean}. Motivated by this problem, we study
the existence of trees whose independence polynomials are symmetric and unimodal.

A finite sequence \(a_0,a_1,\ldots,a_d\) of nonnegative integers is called
\emph{symmetric} if \(a_i=a_{d-i}\) for each \(0\leq i\leq d\). A symmetric sequence
\(a_0,a_1,\ldots,a_d\) is called \emph{unimodal} if \(a_0\leq a_1\leq \cdots \leq a_{\lfloor d/2\rfloor}\). A polynomial \(P(x)=a_0+a_1x+\cdots+a_dx^d\)
with \(a_i\in \mathbb{Z}_{\geq 0}\) is called \emph{symmetric} (resp. \emph{unimodal})
if its coefficient sequence is symmetric (resp. unimodal).

More precisely, we address two existence questions. First, for which integers \(n\) does
there exist a tree on \(n\) vertices whose independence polynomial is symmetric and
unimodal? Second, for which integers \(d\) does there exist a tree whose independence
polynomial is symmetric and unimodal of degree \(d\)?

Our method combines two ingredients. The first is a new gluing construction for rooted
graphs, introduced in the \emph{Bridge Lemma} (Lemma \ref{lem:admissible-gluing}).
This lemma allows us to construct larger rooted trees from smaller ones while preserving
the structure needed in our arguments. The second is the \(\gamma\)-positivity viewpoint
for symmetric polynomials. If \(h(x)\) is symmetric with center of symmetry \(d/2\), then
it admits a unique expansion
\[
h(x)=\sum_{i=0}^{\lfloor d/2\rfloor}\gamma_i x^i(1+x)^{d-2i}.
\]
When all \(\gamma_i\) are nonnegative, the polynomial \(h(x)\) is said to be
\(\gamma\)-positive. This viewpoint provides a convenient route to unimodality; see
\cite{Gamma_positive}. In the body of the paper, we adapt it to rooted trees through the
notion of \(\gamma\)-admissibility.

The notion of $\gamma$-positivity appeared first in the work of D. Foata and M. Schützenberger \cite{gamma_positive_origins} and it is extremely useful as it directly implies symmetry and unimodality. This application of $\gamma$-positivity appears widely in combinatorial and geometric contexts (see \cite{gamma_positive_app,gamma_positive_app_2,gamma_positive_app_3,gamma_positive_app_4,gamma_positive_app_5}).

The main result of the present paper is the following.

\begin{thm}
\label{CHITOSE}
    {\rm (1)} Given \(n \geq 1\) with \(n \notin \{2,4,5,7,10\}\), there is a tree on
    \(n\) vertices whose independence polynomial is symmetric and unimodal.

    {\rm (2)} Given \(d \geq 1\) with \(d \neq 3\), there is a tree whose independence
    polynomial is symmetric and unimodal of degree \(d\).
\end{thm}

\section{Bridge Lemma}
In this section, we introduce a new technique to construct symmetric independence polynomials of trees. Given a graph $G$ and a vertex $r\in V(G)$, the \emph{rooted graph} $(G,r)$ is the graph $G$ with vertex $r$ viewed as a root.  Each vertex of $G$ can be chosen as a root.

\begin{defn}
    Given two rooted graphs $(G,r)$ and $(H,s)$, the \emph{bridging} of $(G,r)$ and  $(H,s)$, denoted by $(G,r)\vee(H,s)$, is the rooted graph with $r$ its root which is obtained from the disjoint union of $(G,r)$ and $(H,s)$ by adding the edge $\{r,s\}$. 
\end{defn}

\begin{remark}
Bridging is not necessarily associative. Let $(A,r), (B,s)$, and $ (C,t)$ be rooted graphs given in \Cref{fig:fig1_0}.  

\begin{figure}[ht]
    \centering
    \includegraphics{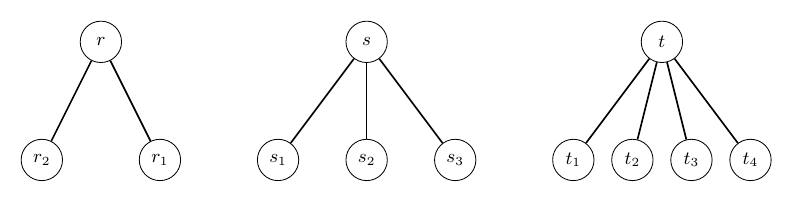}
    \caption{Rooted graphs $(A,r),  (B,s)$ and $ (C,t)$ from left to right}
    \label{fig:fig1_0}
\end{figure}
Notice that 
 \(((A,r)\vee(B,s))\vee(C,t) \neq (A,r)\vee((B,s)\vee(C,t))\) as shown in \Cref{fig:fig1_1}.
\end{remark}

\begin{figure}[ht]
    \centering
    \includegraphics[scale=0.9]{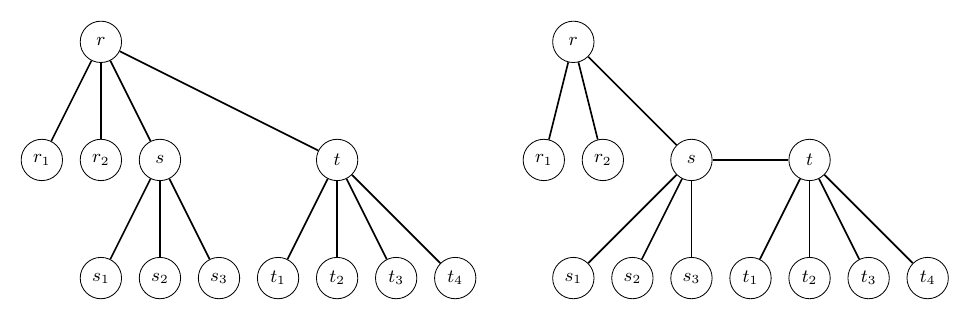}
    \caption{$((A,r)\vee (B,s))\vee (C,t)$ (left) and $(A,r)\vee ((B,s)\vee (C,t))$ (right) }
    \label{fig:fig1_1}
\end{figure}

 To obtain symmetric and unimodal independence polynomials, we use \(\gamma\)-positivity (see \cite{Gamma_positive}. for a survey on $\gamma$-positivity). Although \(\gamma\)-positivity is
usually introduced for symmetric polynomials, for our purposes it is convenient to adopt
the following equivalent formulation. Set
\[
y=\frac{x}{(1+x)^2}.
\]

\begin{defn}\label{def:gamma-positive}
A polynomial \(h(x)\in \mathbb{R}[x]\) is called \emph{\(\gamma\)-positive} if there exist an
integer \(d\geq 0\) and nonnegative real numbers
\(\gamma_0,\gamma_1,\ldots,\gamma_{\lfloor d/2\rfloor}\) such that
\[
h(x)=\sum_{i=0}^{\lfloor d/2\rfloor}\gamma_i x^i(1+x)^{d-2i}.
\]
Equivalently, \(h(x)=(1+x)^d~\Gamma_h(y)\)
where \(\Gamma_h(y)=\sum_{i=0}^{\lfloor d/2\rfloor}\gamma_i y^i \in \mathbb{R}_{\geq 0}[y]\). 
% The polynomial \(\Gamma_h(y)\) is called the \(\gamma\)-polynomial of \(h(x)\).
\end{defn}

The following is well-known in the literature but we include it here for completeness.

\begin{lem}\label{lem:gamma-positive}
If \(h(x)\) is \(\gamma\)-positive, then \(h(x)\) is symmetric and unimodal. If, in
addition, \(h(0)\neq 0\), then \(\deg h(x)=d\).
\end{lem}

\begin{proof}
Since the substitution \(x\mapsto 1/x\) leaves \(y\) invariant, we have
\[
x^d h(1/x)
=
x^d\Bigl(1+\frac1x\Bigr)^d~
\Gamma_h \left(\frac{1/x}{(1+1/x)^2}\right)
=
(1+x)^d~\Gamma_h(y)
=
h(x).
\]
Thus \(h(x)\) is symmetric with center of symmetry \(d/2\).

Moreover, each polynomial \(x^i(1+x)^{d-2i}\) has coefficient sequence obtained from
that of \((1+x)^{d-2i}\) by padding with \(i\) zeros at the beginning and at the end.
Hence each \(x^i(1+x)^{d-2i}\) is symmetric and unimodal with center of symmetry \(d/2\).
Since \(h(x)\) is a nonnegative linear combination of these polynomials, it follows that
\(h(x)\) is unimodal.

Finally, if \(h(0)\neq 0\), then \(\gamma_0=h(0)>0\). Thus the term
\(\gamma_0(1+x)^d\) occurs with positive coefficient. Therefore \(\deg h(x)=d\).
\end{proof}

\begin{remark}
    Let $G$ be any finite simple graph and \(H= G \circ 2K_1\) be the graph obtained by joining each vertex of $G$ to two new vertices. It is known that \(H\) has a symmetric independence polynomial \cite{Stevanovic,hibi2026independencepolynomialsgraphs}. Moreover, it was shown in \cite{2whisker_unimodal} that the independence polynomial of \(H\) is  unimodal. It can also be easily verified that $P_H(x)$ is $\gamma$-positive. This follows from observing that 
    \[
   P_H(x)= \sum_{i=0}^{\alpha(G)} g_i x^i(1+x)^{2n-2i}\]
    where $|V(G)|=n$ and $g_i$ is the number of independent sets of size $i$ in $G$.
\end{remark}

We now introduce the notion of \(\gamma\)-admissibility for rooted graphs.

\begin{defn}\label{def:gamma-admissible}
A rooted graph \((G,r)\) is called \emph{\(\gamma\)-admissible} if there exist an integer
\(d\geq 0\) and polynomials \(A(y),B(y)\in \mathbb{Z}_{\geq 0}[y]\) such that
\[
P_G(x)=(1+x)^dA(y),
\qquad
P_{G-r}(x)=(1+x)^dB(y).
\]
\end{defn}

\begin{ex}
Let \(R_3=P_3\) be rooted at its middle vertex \(r\). Then \((R_3,r)\) is \(\gamma\)-admissible since
\[
P_{R_3}(x)=1+3x+x^2=(1+x)^2(1+y),
\qquad
P_{R_3-r}(x)=1+2x+x^2=(1+x)^2
\]
where  \(A(y)=1+y\) and \(B(y)=1\) such that
\(A(y)-B(y)=y\in \mathbb Z_{\ge 0}[y]\).
\end{ex}

\begin{remark}\label{rem:gamma-admissible}  
If \((G,r)\) is \(\gamma\)-admissible,  one has
\[
x^dP_G(1/x)=P_G(x),\qquad x^dP_{G-r}(1/x)=P_{G-r}(x).
\]
 This implies that both \(P_G(x)\) and \(P_{G-r}(x)\)  have degree $d$ since independence polynomials have constant term \(1\). %Hence both \(P_G(x)\) and \(P_{G-r}(x)\) are symmetric with center of symmetry \(d/2\). 
Since \(P_G(x)\) and \(P_{G-r}(x)\)  are symmetric of degree $d$, by  the standard theory of \(\gamma\)-positivity for symmetric polynomials (see \cite{Gamma_positive}), there are unique coefficients \(a_0,\dots,a_{\lfloor d/2\rfloor} \in \mathbb{Z}_{\geq 0}\) and  \(b_0,\dots,b_{\lfloor d/2\rfloor} \in \mathbb{Z}_{\geq 0}\) 
such that
\[
A(y)=\sum_{i=0}^{\lfloor d/2\rfloor}a_i y^i, \qquad B(y)=\sum_{i=0}^{\lfloor d/2\rfloor}b_i y^i.
\]
 Therefore, \(P_G(x)\) and \(P_{G-r}(x)\) are
\(\gamma\)-positive. Hence, both are symmetric and unimodal by \cref{lem:gamma-positive}.
\end{remark}

The following result, which we refer to as \emph{Bridge Lemma}, is useful in creating symmetric and unimodal independence polynomials for larger graphs from smaller ones.

\begin{lem}{{\bf(Bridge Lemma)}}\label{lem:admissible-gluing}
Let \((T,r)\) and \((U,s)\) be \(\gamma\)-admissible  with
\[
P_T(x)=(1+x)^dA(y), \qquad P_{T-r}(x)=(1+x)^dB(y),
\]
\[
P_U(x)=(1+x)^eC(y), \qquad P_{U-s}(x)=(1+x)^eD(y),
\]
where \(A(y),B(y),C(y),D(y)\in \mathbb Z_{\ge 0}[y]\). Suppose that \(A(y)-B(y)\in \mathbb Z_{\ge 0}[y]\).

If \((W,r)=(T,r)\vee(U,s)\), then 
\((W,r)\) is \(\gamma\)-admissible.
\end{lem}

\begin{proof}
By the standard recursion at the root \(r\), we have
\[
P_W(x)=P_{W-r}(x)+xP_{W-N[r]}(x).
\]
Since \(W-r=(T-r)\sqcup U\) and  \(W-N[r]=(T-N[r])\sqcup(U-s)\),
we obtain
\[
P_W(x)=P_{T-r}(x)P_U(x)+xP_{T-N[r]}(x)P_{U-s}(x).
\]
Since \(xP_{T-N[r]}(x)=P_T(x)-P_{T-r}(x)\),
we have
\[
P_W(x)=P_{T-r}(x)P_U(x)+\bigl(P_T(x)-P_{T-r}(x)\bigr)P_{U-s}(x).
\]
Substituting the \(\gamma\)-admissible expressions gives
\[
P_W(x)=(1+x)^{d+e}\bigl(B(y)C(y)+(A(y)-B(y))D(y)\bigr),
\]
and
\[
P_{W-r}(x)=P_{T-r}(x)P_U(x)=(1+x)^{d+e}B(y)C(y).
\]
Since \(B(y),C(y),D(y)\), and \(A(y)-B(y)\) all belong to \(\mathbb Z_{\ge 0}[y]\),
it follows that \((W,r)\) is \(\gamma\)-admissible. 
\end{proof}

\section{Applications of the Bridge Lemma}

We now apply the Bridge Lemma to rooted trees to construct a tree on $n$ vertices with a symmetric and unimodal independence polynomial as well as the existence of a symmetric and unimodal independence polynomial of degree $d$ for a tree.  

The present section is devoted for a proof of Theorem \ref{CHITOSE}.  Theorem \ref{CHITOSE} (1) follows from Lemmas \ref{lem:tree-large-orders} and \ref{lem:sym_tree_vertex}.  Theorem \ref{CHITOSE} (2) follows from Lemma \ref{lem:tree_sym_alpha}. 

\begin{lem}
\label{lem:tree-large-orders}
Given \(n\geq 19\), there is a tree \(T\) on \(n\) vertices
whose independence polynomial is symmetric and unimodal.
\end{lem}

\begin{proof} 
Let \((R_{19},r)\) be the rooted tree defined as follows. The root \(r\)
has neighbors \(a,b,u,v\), where \(a\) and \(b\) are leaves. The vertex \(u\)
is adjacent to a leaf \(c\) and to a path \(p_1p_2p_3\) through the edge
\(up_1\). The vertex \(v\) is adjacent to four leaves \(d_1,d_2,d_3,d_4\) and
to two vertices \(x\) and \(y\), and each of \(x\) and \(y\) is adjacent to two
leaves, labeled $e_1$, $e_2$ and $f_1$, $f_2$ respectively.

\begin{figure}[ht]
    \centering
    \includegraphics[scale=0.75]{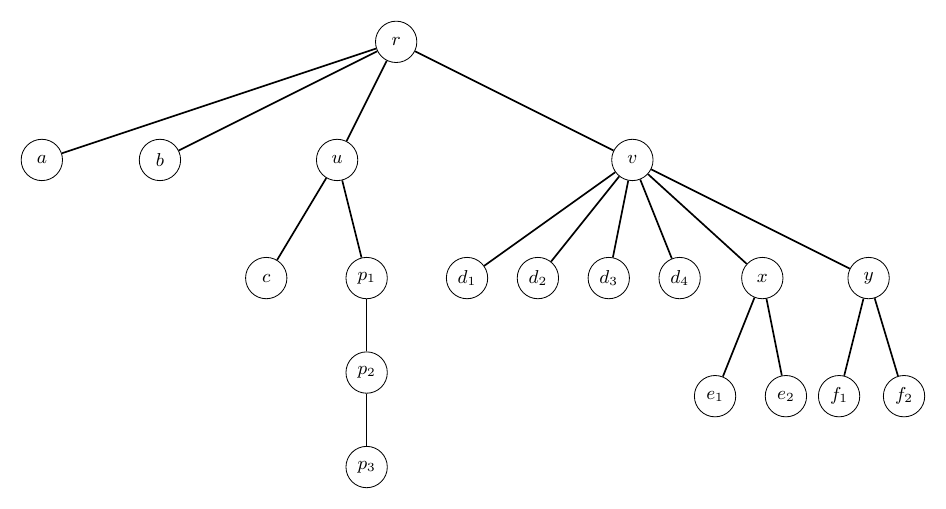}
    \label{fig:R19}
    \caption{The rooted tree $(R_{19},r)$}
\end{figure}

A direct computation gives
\begin{align*}
P_{R_{19}}(x)
&=1+19x+153x^2+701x^3+2058x^4+4112x^5+5772x^6\\
&\qquad +5772x^7+4112x^8+2058x^9+701x^{10}+153x^{11}+19x^{12}+x^{13}\\
&=(1+x)^{13}(1+6y+9y^2+4y^3+y^4).
\end{align*}
and
\begin{align*}
P_{R_{19}-r}(x)
&=1+18x+139x^2+616x^3+1763x^4+3462x^5+4817x^6\\
&\qquad +4817x^7+3462x^8+1763x^9+616x^{10}+139x^{11}+18x^{12}+x^{13}\\
&=(1+x)^{13}(1+5y+6y^2+y^3).
\end{align*}
Hence \((R_{19},r)\) is \(\gamma\)-admissible.

Let \((R_{20},r)\) be the rooted tree defined as follows. The root \(r\)
has neighbors \(a,u,p\), where \(a\) is a leaf. The vertex \(u\) is adjacent to
a leaf \(b\) and to a path \(u_1u_2u_3\) through the edge \(uu_1\). The vertex
\(p\) is adjacent to a vertex \(q\). The vertex \(q\) is adjacent to two leaves
\(c_1,c_2\), to a vertex \(t\) with two leaves $d_1$, $d_2$, and to a vertex \(s\). Finally,
the vertex \(s\) is adjacent to two leaves $e_1, e_2$ and to a vertex \(z\), and \(z\) is
adjacent to two leaves $f_1$, $f_2$.

\begin{figure}[ht]
    \centering
    \includegraphics[scale=0.75]{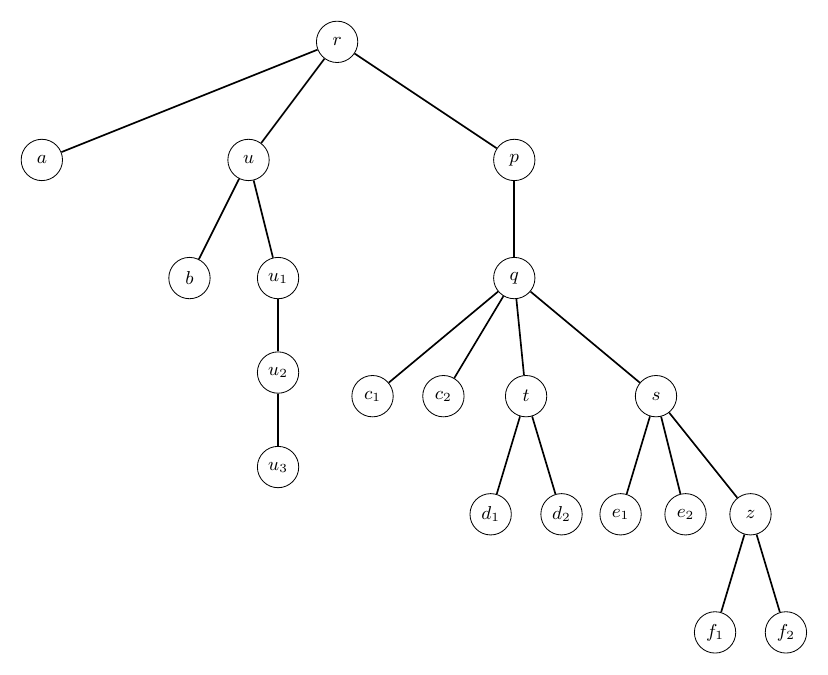}
    \caption{The rooted tree $(R_{20},r)$}
    \label{fig:r20}
\end{figure}

A direct computation gives
\begin{align*}
P_{R_{20}}(x)
&=1+20x+171x^2+829x^3+2548x^4+5255x^5+7496x^6\\
&\qquad +7496x^7+5255x^8+2548x^9+829x^{10}+171x^{11}+20x^{12}+x^{13}\\
& =(1+x)^{13}(1+7y+16y^2+14y^3+4y^4)
\end{align*}
and
\begin{align*}
P_{R_{20}-r}(x)
&=1+19x+155x^2+722x^3+2151x^4+4343x^5+6129x^6\\
&\qquad +6129x^7+4343x^8+2151x^9+722x^{10}+155x^{11}+19x^{12}+x^{13}\\
&=(1+x)^{13}(1+6y+11y^2+7y^3+y^4).
\end{align*}
Hence \((R_{20},r)\) is \(\gamma\)-admissible.

Now fix \(n\geq 19\).   Set \((T_0,\rho_0)=
\begin{cases}
(R_3,r), & \text{if }  n \equiv 0 \pmod 3,\\[2mm]
(R_{19},r), & \text{if }  n \equiv 1 \pmod 3,\\[2mm]
(R_{20},r), & \text{if }  n \equiv 2 \pmod 3.
\end{cases}\).

Then \(|V(T_0)|\equiv n \pmod 3\) and \(n=|V(T_0)|+3m\)
for some integer \(m\geq 0\).

For \(1\leq j\leq m\), define recursively \((T_j,\rho_j)=(R_3,r)\vee (T_{j-1},\rho_{j-1})\). Since \((R_3,r)\) and \((T_{j-1},\rho_{j-1})\) are \(\gamma\)-admissible and
\(A(y)-B(y)=y\in \mathbb Z_{\geq 0}[y]\) for \((R_3,r)\), repeated applications of the Bridge Lemma
(\Cref{lem:admissible-gluing}) show that each \((T_j,\rho_j)\) is \(\gamma\)-admissible.

Each bridging step adds exactly three vertices, so \(|V(T_m)|=|V(T_0)|+3m=n\).
Therefore \(T_m\) is a tree on \(n\) vertices. In each case (\(n  \pmod 3\)),   the resulting tree has symmetric and unimodal independence polynomial by \Cref{rem:gamma-admissible}.
\end{proof}

\begin{remark}
A computer search using Macaulay2 over all nonisomorphic trees on at most \(21\) vertices shows
the following. The rooted path \(R_3=P_3\), rooted at its middle vertex, is the
smallest $\gamma$-admissible rooted tree. Repeated applications of
\cref{lem:admissible-gluing} to \(R_3\) yield $\gamma$-admissible rooted trees on
\(6,9,12,15,\) and \(18\) vertices.   On the other hand, among orders not
divisible by \(3\), there are no $\gamma$-admissible rooted trees for \(4 \leq n \leq 18\).

At \(n=19\), there is, up to isomorphism, a unique $\gamma$-admissible tree. Moreover, among the \(317,955\) nonisomorphic
trees on \(19\) vertices, exactly \(6\) have symmetric  and unimodal independence polynomial.
Likewise, at  \(n=20\), there is, up to isomorphism, a unique $\gamma$-admissible
tree. Among the \(823,065\)
nonisomorphic trees on \(20\) vertices, exactly \(14\) have symmetric and unimodal
independence polynomial. Among the 2,144,505 nonisomorphic trees on \(21\) vertices, exactly 22 have symmetric and unimodal independence polynomial, of which 14 are $\gamma$-admissible.  
\end{remark}

\begin{center}
\begin{tabular}{|c|c|c|c|}
\hline
    $n$ & \# of non-isomorphic trees & \# of symmetric trees & \# of $\gamma$-admissible trees  \\
    \hline
    3 & 1 & 1 & 1  \\
    \hline
    4 & 2 & 0 & 0 \\
    \hline 
    5 & 3 & 0 & 0 \\
    \hline
    6 & 6 & 1 & 1 \\
    \hline
    7 & 11 & 0 & 0 \\
    \hline
    8 & 23 & 1 & 0 \\
    \hline
    9 & 47 & 1 & 1 \\
    \hline
    10 & 106 & 0 & 0  \\
    \hline
    11 & 235 & 2 & 0 \\
    \hline
    12 & 551 & 3 & 2 \\
    \hline
    13 & 1301 & 3 & 0 \\
    \hline
    14 & 3159 & 1 & 0 \\
    \hline
    15 & 7741 & 4 & 3 \\
    \hline
    16 & 19320 & 2 & 0 \\
    \hline
    17 & 48629 & 4 & 0 \\
    \hline 
    18 & 123867 & 15 & 6 \\
    \hline 
    19 & 317955 & 6 & 1 \\
    \hline
    20 & 823065 & 14 & 1 \\
    \hline 
    21 & 2144505 & 22 & 14\\
    \hline
\end{tabular}
\end{center}

Given nonnegative integers \(a_1,\dots,a_m\), we introduce the graph  
\[
C(a_1,\dots,a_m)
\]
which is 
obtained from a path \(v_1\cdots v_m\) by attaching
\(a_i\) leaves to \(v_i\) for each \(i\).

\begin{figure}[ht]
    \centering
    \includegraphics[width=0.7\linewidth]{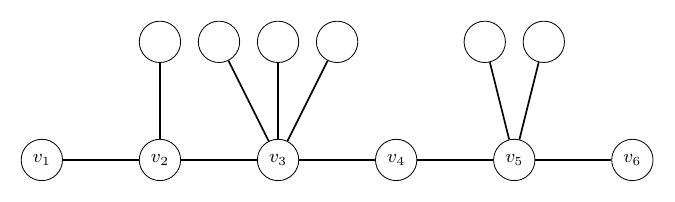}
    \caption{The graph $C(0,1,3,0,2,0)=C(2,3,0,3)$}
    \label{fig:caterpillar}
\end{figure}

\begin{lem}\label{lem:sym_tree_vertex}
  Let $n\in [1,18]$.  There is a  tree \(T\) on \(n\) vertices such that \(P_T(x)\) is symmetric and unimodal when \(n\notin\{2,4,5,7,10\}\). When  $n\in \{2,4,5,7,10\}$, there is no tree with a symmetric  independence polynomial.
\end{lem}

\begin{proof}
For \(n=2\), the only tree is \(P_2\) where \(P_{P_2}(x)=1+2x\)
is not symmetric.

For \(n\in\{4,5,7,10\}\), it was verified by a finite computer check using Macaulay2 that for each tree \(T\) on \(n\) vertices \(P_T(x)\) is never symmetric.

For \(n=1\), take \(G=K_1\). Then $P_G(x)=1+x$, which is symmetric and unimodal.

For \(n=3\), take \(G=P_3\). Then
\(P_G(x)=1+3x+x^2\), which is symmetric and unimodal.

For \(n=6,9,12,15,18\),  repeated applications of the Bridge Lemma
(\cref{lem:admissible-gluing}) to the rooted path \((R_3,r)\) with its middle vertex $r$
as root yield trees with symmetric and unimodal independence polynomials.

For the remaining $n$, one may take
\[
G_8:=C(1,0,0,0,2), \quad
G_{11}:=C(1,0,1,2,0,1), \quad
G_{13}:=C(1,0,1,4,0,1),
\]
\[
G_{14}:=C(2,1,0,0,0,2,2), \quad
G_{16}:=C(2,1,0,1,0,1,1,0,1), \quad
G_{17}:=C(1,0,0,0,2,3,1,0,1).
\]

A direct computation of the independence polynomials (given below) of the above caterpillars gives symmetric and unimodal independence polynomials.
\begin{align*}
P_{G_8}(x)
&=1+8x+21x^2+21x^3+8x^4+x^5,\\
P_{G_{11}}(x)
&=1+11x+45x^2+88x^3+88x^4+45x^5+11x^6+x^7,\\
P_{G_{13}}(x)
&=1+13x+66x^2+176x^3+279x^4+279x^5+176x^6+66x^7+13x^8+x^9,\\
P_{G_{14}}(x)
&=1+14x+78x^2+226x^3+377x^4+377x^5+226x^6+78x^7+14x^8+x^9,\\
P_{G_{16}}(x)
&=1+16x+105x^2+369x^3+764x^4+970x^5\\
&\qquad  \qquad \qquad \qquad \qquad \qquad +764x^6+369x^7+105x^8+16x^9+x^{10},\\
P_{G_{17}}(x)
&=1+17x+120x^2+465x^3+1101x^4+1676x^5+1676x^6\\
&\qquad \qquad \qquad \qquad \qquad \qquad +1101x^7+465x^8+120x^9+17x^{10}+x^{11}. \qedhere
\end{align*}
\end{proof}

Lastly, we fix the degree of an independence polynomial and construct trees with symmetric and unimodal independence polynomial of given degree.

\begin{lem}\label{lem:tree_sym_alpha}
   Given $d\geq 1$ with $d\neq 3$, there is a tree $T$ such that $P_T(x)$ is symmetric and unimodal with $\deg P_T(x)=d$.
\end{lem}

\begin{proof}
We first show that degree \(3\) is impossible. Suppose that \(T\) is a tree on
\(n\) vertices such that \(P_T(x)\) is symmetric and \(\deg P_T(x)=3\). Then
\[
P_T(x)=1+nx+g_2x^2+x^3.
\]
Since \(P_T(x)\) is symmetric, we must have \(g_2=n\). On the other hand, a tree
on \(n\) vertices has exactly \(n-1\) edges. So the number of independent sets
of cardinality \(2\) is
\[
g_2=\binom{n}{2}-(n-1)=n,
\]
 that is, \(n^2-5n+2=0,\)
which has no integer solution. Thus there is no tree \(T\) such that
\(P_T(x)\) is symmetric and \(\deg P_T(x)=3\).

 For \(d=1\), take \(T=K_1\). Then \(P_T(x)=1+x\). Now let \(d\ge 2\) with \(d\neq 3\).

If \(d\) is even, write \(d=2k\) with \(k\ge 1\). Starting from \((R_3,r)\) and
applying \Cref{lem:admissible-gluing} using \(k\)-copies of \((R_3,r)\), we obtain a tree whose independence
polynomial is symmetric and unimodal of degree \(2k=d\).

If \(d\in\{5,7,9,11\}\), one may take \(G_8,G_{11},G_{14},G_{17}\), respectively.

Finally, let \(d\ge 13\) be odd. Write \(d=13+2m\) with \(m\ge 0\). Starting from
\((R_{19},r)\) and repeatedly bridging \(m\) copies of \((R_3,r)\) on the left,
\Cref{lem:admissible-gluing} yields a tree whose independence polynomial is symmetric
and unimodal of degree \(13+2m=d\). 
\end{proof}

\section*{Closing Remarks}

A supplementary catalogue accompanying this paper records the finite computations behind the small-order counts. It lists every nonisomorphic tree on at most $21$ vertices whose independence polynomial is symmetric (and unimodal), together with the corresponding independence polynomial. It also records, for each $\gamma$-admissible tree, the factorization
\[
P_T(x)=(1+x)^dA(y), \qquad P_{T-r}(x)=(1+x)^dB(y), \qquad y=\frac{x}{(1+x)^2},
\]
for every admissible root.

\bibliographystyle{abbrv}
\bibliography{ref}

\end{document}